\documentclass[11pt]{article}
\usepackage{amsmath,amsthm,amsfonts,amssymb,mathrsfs,bm}
\usepackage{amsmath,amsthm,amsfonts,amssymb,bm,wasysym}
\usepackage{epsfig}
\usepackage[usenames]{color}
\usepackage{verbatim}
\usepackage{hyperref}
\usepackage{multicol}
\usepackage{comment}
\usepackage{float}
\usepackage{graphicx}
\usepackage{centernot}
\usepackage{tikz}
\usepackage{color}
\usepackage[]{algorithm2e}
\usepackage{enumerate}

\graphicspath{{./Pics}}
\usepackage[color=green, textsize=tiny]{todonotes}

\usepackage[normalem]{ulem}

\parskip \medskipamount
\newtheorem{theorem}{Theorem}[section]
\newtheorem{definition}[theorem]{Definition}

\numberwithin{equation}{section}
\newtheorem{lemma}[theorem]{Lemma}

\newtheorem{remark}[theorem]{Remark}

\newtheorem{claim}[theorem]{Claim}

\numberwithin{equation}{section}

\def\N{\mathbb{N}}
\def\Z{\mathbb{Z}}

\newcommand{ \cL}{ \mathcal L }

\renewcommand{\phi}{\varphi}
\renewcommand{\epsilon}{\varepsilon}

\allowdisplaybreaks

\newcommand{\1}{{\text{\Large $\mathfrak 1$}}}

\renewcommand{\emptyset}{\varnothing}

\newcommand{\pr}[1]{\mathbb{P}\!\left(#1\right)}
\newcommand{\E}[1]{\mathbb{E}\!\left[#1\right]}

\newcommand{\prstart}[2]{\mathbb{P}_{#2}\!\left(#1\right)}

\newcommand{\bcap}[1]{{\mathrm{BCap}}(#1)}

\def\cR{\mathcal{R}}

\def\cF{\mathcal{F}}

\newcommand{\tn}{|\kern-.1em|\kern-0.1em|}

\newcommand\be{\begin{equation}}
	\newcommand\ee{\end{equation}}

\newcommand{\cT}{\mathcal{T}}

\newcommand{\musb}{\mu_{\mathrm{sb}}}

\newcommand{\ball}[2]{B(#1,#2)}

\begin{document}
	
	\title{\bf Derivative formula for capacities}
	
	\author{Amine Asselah \thanks{
			Universit\'e Paris-Est, LAMA, UMR 8050, UPEC, UPEMLV, CNRS, F-94010 Cr\'eteil; amine.asselah@u-pec.fr} \and
		Bruno Schapira\thanks{Universit\'e Claude Bernard Lyon 1, Institut Camille Jordan, CNRS UMR 5208, 43 Boulevard du 11 novembre 1918, 69622 Villeurbanne Cedex, France;  schapira@math.univ-lyon1.fr} \and Perla Sousi\thanks{University of Cambridge, Cambridge, UK;   p.sousi@statslab.cam.ac.uk} 
	}
	\date{}
	\maketitle
	
	\begin{abstract}  
		We obtain a derivative formula for various notions of capacity. Namely we identify the second order term in the asymptotic expansion of the capacity of a union of two sets, as their distance goes to infinity. Our result applies to the usual Newtonian capacity in the setting of random walks on the Euclidean lattice, to the family of Bessel-Riesz capacities, and to the Branching capacity, which has been introduced recently by Zhu~\cite{Zhu} in connection with critical Branching random walks. On the other hand, the result remains open for the notion of capacity in the setting of percolation, which is introduced in a companion paper, but serves as a motivation, as it would have some interesting consequences there.

		\bigskip
		\noindent \emph{Keywords and phrases.} Newtonian capacity, Bessel-Riesz capacity, Branching capacity.

		\noindent MSC 2010 \emph{subject classifications.} Primary 60J45, Secondary 60J80, 31C20.
	\end{abstract}

	\section{Introduction}
	Our aim in this paper is to prove some derivative formula for various notions of capacities, which were motivated by our recent study of the notion of capacity in high dimensional critical percolation~\cite{ASS25}. Curiously these formulas, despite being very simple, seem to be new, to the best of our knowledge.

	Let us start with a simple example, namely the Newtonian capacity appearing in the setting of random walks. We denote by $(S_n)_{n\ge 0}$ a simple random walk started from $0$ on $\mathbb Z^d$, with $d\ge 3$, and by $g(x) = \sum_{n\ge 0} \mathbb P(S_n=x)$ its associated Green's function. It is well-known, see~\cite{LL10}, that the following limit exists and defines the Newtonian capacity of any finite (and nonempty) set $A\subset \mathbb Z^d$, 
	\begin{equation}\label{def.New.cap}
		\textrm{Cap}(A) = \lim_{\|z\|\to \infty} \frac{\mathbb P(\mathcal (z+R_\infty) \cap A\neq \emptyset)}{g(z)}, 
	\end{equation}
	where $\mathcal R_\infty = \{S_0,S_1,\dots\}$ is the range of the walk. 
	In this setting, our result states that for any finite sets $A,B\subset \mathbb Z^d$,  
	\begin{equation}\label{limit.NewCap}
		\lim_{\|z\|\to \infty}\frac{ \textrm{Cap}(A) +  \textrm{Cap}(B) -  \textrm{Cap}(A \cup(z + B))}{g(z)} = 2 \cdot  \textrm{Cap}(A)\cdot \textrm{Cap}(B).
	\end{equation}
	The second example we treat is the family of the so-called Bessel-Riesz capacities. 
	Recall that they are defined for any $\alpha\in (0,d)$, by 
	\begin{equation}\label{def.capalpha}
		\textrm{Cap}_\alpha (A) = \Big (\inf \big\{\sum_{x,y\in A} g_\alpha(y-x) \mu(x)\mu(y) : \mu \textrm{ probability measure on }A\big\}\Big)^{-1},
	\end{equation}
	where $g_\alpha(z) = (1+\|z\|)^{-\alpha}$ with $\|\cdot \|$ denoting the Euclidean norm. As above, we prove that for any $\alpha \in (0,d)$, and any finite $A,B\subset \mathbb Z^d$, 
	\begin{equation}\label{limit.RieszCap}
		\lim_{\|z\|\to \infty}\frac{ \textrm{Cap}_\alpha(A) +  \textrm{Cap}_\alpha(B) -  \textrm{Cap}_\alpha(A \cup(z + B))}{g_\alpha(z)} = 2 \cdot  \textrm{Cap}_\alpha(A)\cdot \textrm{Cap}_\alpha(B).
	\end{equation}	
	Finally the third example we consider is the recently introduced notion of branching capacity. 
	To define it, let $\cT_c$ be a critical Bienaym\'e-Galton-Watson tree, i.e.\ the offspring distribution has mean~$1$ and finite variance $\sigma^2<\infty$. We assign i.i.d.\ simple random walk in $\Z^d$ increments to the edges of~$\cT_c$. The branching random walk indexed by $\cT_c$ and started from $z\in \Z^d$ is the process $(S_u)_{u\in \mathcal T_c}$ defined as follows: $S_\emptyset=z$, where $\emptyset$ denotes the root of $\cT_c$ and for any other vertex $u\in \cT_c$ the value $S_u$ is defined as the sum of $z$ plus the sum of the increments on the edges of the tree along the shortest path from $\emptyset$ to $u$. 
	We denote by $\mathcal T_c^z=\{S_u :  u\in \mathcal T_c\}$ its range. It has been proved by Zhu~\cite{Zhu} that the following limit exists for any finite $A\subset\mathbb Z^d$, and defines the branching capacity of~$A$: 
	\begin{equation}\label{BCaphit}
		\textrm{BCap}(A)= \lim_{\|z\|\to \infty}\frac{\mathbb P(\mathcal T^z_c \cap A \neq \emptyset ) }{g(z)}. 
	\end{equation}
	We stress that this notion of branching capacity has proven to be a fundamental tool in recent studies of critical branching random walks~\cite{ASS23,ASS25,AOSS25,Sch, Zhu, Zhu2,Zhu3,Zhu4}.  
	Our result is that for any finite sets $A$ and $B$, 
	\begin{equation}\label{limit.BCap}
		\lim_{\|z\|\to \infty}\frac{ \textrm{BCap}(A) +  \textrm{BCap}(B) -  \textrm{BCap}(A \cup(z + B))}{G(z)} = 2 \cdot  \textrm{BCap}(A)\cdot \textrm{BCap}(B), 
	\end{equation}
	where the function $G$ will be defined later, in Section~\ref{sec.branchingcap}, and satisfies the following asymptotic:
	\begin{equation}\label{asympG}
		G(z) \sim c_d \cdot \|z\|^{4-d},
	\end{equation} 
	for some constant $c_d>0$, see e.g.~\cite{ASS23}. 
	To conclude we note that a similar result has been conjectured in the setting of percolation capacity in~\cite{ASS25}. In particular, if true, it would have interesting consequences on the asymptotic probability that the so-called Incipient Infinite Cluster (which is a critical percolation cluster conditioned on being infinite) intersects a finite set.

	Interestingly, our proofs of the three results~\eqref{limit.NewCap}, \eqref{limit.RieszCap} and~\eqref{limit.BCap} rely on different arguments. Namely for the proof of~\eqref{limit.NewCap}, in the setting of Newtonian capacity, which is done in Section~\ref{sec.Newton}, we rely on an exact expression of the capacity of a union of two sets in terms of a ``cross term", which was introduced in~\cite{ASS19}. 
	Concerning the proof of~\eqref{limit.RieszCap}, which is done in Section~\ref{sec.Rieszcap}, we make use of various equivalent variational formulas defining them. Finally for the proof of~\eqref{limit.BCap} in Section~\ref{sec.branchingcap}, we rely on specific properties of branching capacity, in particular we use that equilibrium measure can be expressed both as an escape probability and a harmonic measure from infinity, as established by~Zhu~\cite{Zhu}.

	
	\section{Case of the Newtonian capacity}\label{sec.Newton}
	Fix $A$ a finite subset of $\mathbb Z^d$, with $d\ge 3$. 
	Denote here by $\mathbb P_x$ the law of a simple random walk $(S_n)_{n\ge 0}$ on $\mathbb Z^d$ starting from $x$ (abbreviated in $\mathbb P$ when the random walk starts from the origin), and by $H_A=\inf\{n\ge 0 : S_n \in A\}$, the hitting time of $A$. We shall also write $H_A^+ = \inf\{n\ge 1 :S_n \in A\}$ for the first return time to $A$. 
	It was proved in~\cite[Proposition 1.6]{ASS19} that for any finite $A,B\subset \mathbb Z^d$, 
	$$\textrm{Cap}(A\cup B) = \textrm{Cap}(A) + \textrm{Cap}(B) - \chi(A,B) - \chi(B,A)  + \varepsilon(A,B),$$
	where 
	$$\chi(A,B) = \sum_{x\in A} \sum_{y\in B} \mathbb P_x(H_{A\cup B}^+ = \infty) g(y-x) \mathbb P_y(H_B^+ = \infty), $$
	and $0\le \varepsilon(A,B) \le \textrm{Cap}(A\cap B)$. 
	In particular, for any $\|z\|$ large enough, one has $A\cap (z+B) = \emptyset$, and consequently $\varepsilon(A,z+B) = 0$. Moreover, for any $x\in A$, one has $\lim_{\|z\|\to \infty} \mathbb P_x(H_{z+B} <\infty) = 0$, and thus also $\lim_{\|z\|\to \infty} \mathbb P_x(H_{A\cup (z+ B)}^+ = \infty) = \mathbb P_x(H_A^+= \infty)$. 
	This entails 
	$$\lim_{\|z\|\to \infty} \frac{\chi(A,z+B)}{g(z)} = \textrm{Cap}(A)\cdot \textrm{Cap}(B),$$
	and~\eqref{limit.NewCap} follows.

	Note that we could also have argued more directly as follows. First, by definition of the Newtonian capacity~\eqref{def.New.cap}, one has for any $z\in \mathbb Z^d$, 
	$$\textrm{Cap}(A) + \textrm{Cap}(B) - \textrm{Cap}(A\cup(z+B) =  \lim_{\|w\|\to \infty} \frac{\mathbb P_w(H_A<\infty, H_{z+B}<\infty)}{g(w)}.$$ 
	Next, observe that if $\|z\|$ is large enough so that $A\cap (z+B) = \emptyset$, one has 
	$$\mathbb P_w(H_A<\infty, H_{z+B}<\infty) = \mathbb P_w(H_A<H_{z+B}<\infty) + \mathbb P_w(H_{z+B}<H_A<\infty).$$
	We can then use the Markov property, and write, 
	$$\mathbb P_w(H_A<H_{z+B}<\infty) = \sum_{a\in A} \mathbb P_w(H_A<H_{z+B}, S_{H_A} =a) \cdot \mathbb P_a(H_{z+B}<\infty).$$
	Now it is known~\cite{LL10} that $g(z) =(1+o(1)) \|z\|^{2-d}$, and thus uniformly in $a\in A$, 
	$$\mathbb P_a(H_{z+B}<\infty) = \textrm{Cap}(B)\cdot g(z) + o(g(z)),$$ 
	and hence plugging this above we infer 
	$$\mathbb P_w(H_A<H_{z+B}<\infty) = \mathbb P_w(H_A<H_{z+B}) \cdot \big(\textrm{Cap}(B) \cdot g(z) + o(g(z))\big). $$
	Now, observing that $\mathbb P_w(H_A<H_{z+B}) = \mathbb P_w(H_A<\infty) - \mathcal O(g(w-z) g(z))$, we finally deduce that 
	$$\lim_{\|z\|\to \infty} \frac 1{g(z)} \lim_{\|w\|\to \infty} \frac{\mathbb P_w(H_A<H_{z+B}<\infty) }{g(w)} = \textrm{Cap}(A) \cdot \textrm{Cap}(B),$$
	whence the result. 
	When dealing with the branching capacity later, we will follow a similar strategy, though more complicated.


	\section{Case of general Bessel-Riesz capacities}\label{sec.Rieszcap}
	We prove here the result for general Bessel-Riesz capacities. It is folklore that $\textrm{Cap}_\alpha(A)$ satisfies
	\begin{align}\label{def.capalpha2}
		\textrm{Cap}_\alpha(A)  = \sup_{\varphi : A \to \mathbb R_+} \Big\{\sum_{x\in A} \varphi(x) : \sup_{x\in A} g_\alpha * \varphi(x) \le 1\Big\} =  \inf_{\varphi : A \to \mathbb R_+}\Big\{\sum_{x\in A} \varphi(x) :  \inf_{x\in A} g_\alpha * \varphi(x) \ge 1\Big\}, 
	\end{align}
	see e.g.~the Appendix in~\cite{AS24} for a proof of the first equality, together with the fact that if $\mu$ is a probability measure realizing the infimum in~\eqref{def.capalpha}, then the function $\varphi$ defined by 
	$\varphi(x) = \textrm{Cap}_\alpha(A) \cdot \mu(x)$, satisfies $g_\alpha * \varphi(x) = 1$, for all $x\in A$. From this, the second equality in~\eqref{def.capalpha2} follows, see e.g.~\cite{DRS}.

	Now let $A$ and $B$ be finite subsets of $\mathbb Z^d$, and let $z\in \mathbb Z^d$, be such that $A \cap (z+B) = \emptyset$. 
	Let $\mu_A$ and $\mu_B$ be probability measures realizing the infimum in~\eqref{def.capalpha}, respectively for $\textrm{Cap}_\alpha(A)$ and 
	$\textrm{Cap}_\alpha(B)$, and let $\varphi_A$ and $\varphi_B$ be defined by $\varphi_A(x) = \textrm{Cap}_\alpha(A)\cdot \mu_A(x)$,  and $\varphi_B(x) = \textrm{Cap}_\alpha(B)\cdot \mu_B(x)$, respectively (note that $\varphi_A$ is equal to zero outside $A$ and similarly for $\varphi_B$). 
	Consider $\varphi$ the function defined for $x\in \mathbb Z^d$, by 
	$$\varphi(x) = \varphi_A(x) + \varphi_B(x-z). $$ 
	Note that by definition, 
	$$\sum_{x\in \mathbb Z^d} \varphi(x) = \sum_{x\in A \cup (z+B)} \varphi(x) = \textrm{Cap}_\alpha(A)+ \textrm{Cap}_\alpha(B).$$ 
	Let then $\mu$ be the probability measure on $A\cup (z+B)$ defined by 
	$$\mu(x) = \frac{\varphi(x)}{\textrm{Cap}_\alpha(A)+ \textrm{Cap}_\alpha(B)}.$$  
	One has by~\eqref{def.capalpha}, 
	$$\textrm{Cap}_\alpha(A\cup (z+B)) \ge \frac 1{\sum_{x,y \in A\cup (z+B)} g_\alpha(y-x) \mu(x) \mu(y)},$$
	and using the definition of $\varphi_A$ and $\varphi_B$, we get that as $\|z\|\to \infty$, 
	\begin{align*}
		& \sum_{x,y \in A\cup (z+B)} g_\alpha(y-x) \mu(x) \mu(y) \\
		& = \frac{\textrm{Cap}_\alpha(A)}{(\textrm{Cap}_\alpha(A)+ \textrm{Cap}_\alpha(B))^2}
		+ \frac{\textrm{Cap}_\alpha(B)}{(\textrm{Cap}_\alpha(A)+ \textrm{Cap}_\alpha(B))^2} + 2\sum_{x\in A}\sum_{y \in z+B} g_\alpha(y-x) \mu(x) \mu(y)\\
		& = \frac{1}{\textrm{Cap}_\alpha(A)+ \textrm{Cap}_\alpha(B)}
		+ 2(1+o(1)) g_\alpha(z)\frac{\textrm{Cap}_\alpha(A)\cdot\textrm{Cap}_\alpha(B)}{(\textrm{Cap}_\alpha(A)+ \textrm{Cap}_\alpha(B))^2}. 
	\end{align*}
	Consequently, we deduce that 
	\begin{equation}\label{eq:capalpha1}
		\limsup_{\|z\|\to \infty} \frac{\textrm{Cap}_\alpha(A)+\textrm{Cap}_\alpha(B) - \textrm{Cap}_\alpha(A\cup (z+B))}{g_\alpha(z)} \le 2\cdot \textrm{Cap}_\alpha(A)\cdot\textrm{Cap}_\alpha(B). 
	\end{equation}
	For the other direction, fix some $\varepsilon \in (0,1)$, let 
	$$a = 1 - (1-\varepsilon)g_\alpha(z) \textrm{Cap}_\alpha(B), \quad \text{and}\quad b =  1 - (1-\varepsilon)g_\alpha(z) \textrm{Cap}_\alpha(A),$$
	and let for $x\in \mathbb Z^d$, 
	$$\psi (x) = a \cdot \varphi_A(x)  + b \cdot \varphi_B(x-z). $$ 
	We claim that when $\|z\|$ is large enough, one has 
	\begin{equation}\label{claim} 
		\inf_{x\in A \cup (z+B)} g_\alpha * \psi(x) \ge 1.
	\end{equation}
	Indeed, recall that $g_\alpha * \varphi_A(x) =1$ for all $x\in A$, and $g_\alpha * \varphi_B(x-z) =(1+o(1)) g_\alpha(z) \cdot \textrm{Cap}_\alpha(B)$, as $\|z\|\to \infty$, which entails 
	$$g_\alpha *\psi(x) = 1+ \varepsilon g_\alpha(z) \cdot \textrm{Cap}_\alpha(B) + o(g_\alpha(z)),$$ 
	as $\|z\|\to \infty$ for all $x\in A$, whence we deduce $\inf_{x\in A} g_\alpha * \psi(x) \ge 1$, for $\|z\|$ large enough. A similar argument shows as well that $\inf_{x\in z+B} g_\alpha * \psi(x) \ge 1$, for $\|z\|$ large enough, proving the claim~\eqref{claim}. Then by the second equality in~\eqref{def.capalpha2}, we get that for $\|z\|$ large enough, 
	\begin{align*}
		\textrm{Cap}_\alpha(A\cup (z+B)) & \le \sum_{x\in A\cup (z+B)} \psi(x) = a \cdot \textrm{Cap}_\alpha(A) + b \cdot \textrm{Cap}_\alpha(B)\\
		& = \textrm{Cap}_\alpha(A) + \textrm{Cap}_\alpha(B) - 2 (1- \varepsilon) g_\alpha(z) \cdot \textrm{Cap}_\alpha(A) \cdot \textrm{Cap}_\alpha(B). 
	\end{align*}
	Since this holds for any $\varepsilon>0$, we deduce that 
	\begin{equation}\label{eq:capalpha2}
		\liminf_{\|z\|\to \infty} \frac{\textrm{Cap}_\alpha(A)+\textrm{Cap}_\alpha(B) - \textrm{Cap}_\alpha(A\cup (z+B))}{g_\alpha(z)} \ge 2\cdot \textrm{Cap}_\alpha(A)\cdot\textrm{Cap}_\alpha(B). 
	\end{equation}
	Combining~\eqref{eq:capalpha1} and~\eqref{eq:capalpha2} yields
	$$
	\lim_{\|z\|\to \infty} \frac{\textrm{Cap}_\alpha(A)+\textrm{Cap}_\alpha(B) - \textrm{Cap}_\alpha(A\cup (z+B))}{g_\alpha(z)} = 2\cdot \textrm{Cap}_\alpha(A)\cdot\textrm{Cap}_\alpha(B). 
	$$

	
	\section{Case of the Branching capacity} 
	\label{sec.branchingcap} We first recall all the necessary definitions.
	So we consider $\mu$ a probability measure on the integers with mean one, and a finite third moment. We let $\mathcal T_c$ be a planar Bienaym\'e-Galton-Watson tree with offspring distribution $\mu$. 
	We shall also consider the infinite invariant tree $\mathcal T$. To define it, consider the measure $\musb$ defined by $\musb(i)= i\mu(i)$ for $i\in \mathbb N$.  Then $\mathcal T$ is a planar rooted tree defined as follows: 
	\begin{itemize}
		\item The root produces $i$ offspring with probability $\mu(i-1)$ for every $i\ge 1$. The first offspring of the root is \emph{special}, while the others if they exist are \emph{normal}. 
		\item Special vertices produce offspring independently according to $\musb$, while normal vertices produce offspring independently according to $\mu$. 
		\item Each special vertex produces exactly one special vertex chosen uniformly at random among its children, while the other children are normal.
	\end{itemize}
	The set of special vertices is called the spine of the tree, and is a copy of $\mathbb N$. The set of vertices on the left of the spine, including those on the spine, is called the past of $\mathcal T$ and denoted $\mathcal T_-$ (note that by definition the root is not part of the spine, and thus not part of $\mathcal T_-$ neither). We refer e.g. to~\cite{ASS23} for more details on these trees, in particular their fundamental property of invariance by rerooting. Letting $\widetilde \mu$ be defined by 
	$\widetilde \mu(i) = \sum_{j>i} \mu(j)$, the number of normal offspring of any special vertex, which belong to the left of the spine is distributed according to $\widetilde \mu$, and similarly for the number of normal offspring on the right of the spine. A tree whose root has a number of offspring distributed according to $\widetilde \mu$ and all other vertices according to $\mu$ is denoted $\widetilde{\mathcal T}_c$ and is called an adjoint tree. In particular by definition, the trees hanging off the spine (on its right or on its left) are distributed as $\widetilde{\mathcal T}_c$. Similarly a tree whose root has a number of offspring distributed according to $\musb$ minus one, and all other vertices according to $\mu$ is denoted $\widehat{\mathcal T}_c$. 
	
	The definition of a random walk indexed by a critical tree as given in the introduction generalises immediately to any general tree~$T$. We denote its range when started from $x$ by~$T^x$. 
	
	We now consider the simple random walk $(S_u)_{u\in \mathcal T}$ started from $0$ indexed by $\mathcal T$ and define the associated Green's function $G$, by setting for any $z\in \mathbb Z^d$
	\[
	G(z) =\E{\sum_{u\in \cT_-}\1(S_u=z)}.
	\]
   
	It was proved in~\cite{Zhu} that when $d\ge 5$, for any finite $A\subset \mathbb Z^d$, the branching capacity of $A$ satisfies 
	$$\textrm{BCap}(A) = \lim_{\|w\|\to \infty} \frac{\mathbb P(\mathcal T_c^w\cap A\neq \emptyset)}{g(w)}= \lim_{\|w\|\to \infty} \frac{\mathbb P(\mathcal T_-^w\cap A\neq \emptyset)}{G(w)} = \sum_{x\in A} \mathbb P(\mathcal T^x_- \cap A = \emptyset).$$  
	It follows in particular, using the exclusion-inclusion formula, that for any $z\in \mathbb Z^d$, and any finite $A,B\subset \mathbb Z^d$, 
	\begin{align}\label{eq:incexcl}
	\textrm{BCap}(A) + \textrm{BCap}(B) - \textrm{BCap}(A\cup (z+B))  = \lim_{\|w\|\to \infty} \frac{\mathbb P( \mathcal T^w_c \cap A\neq \emptyset, 
		\mathcal T^w_c \cap (z+B)\neq \emptyset)}{g(w)}. 
		\end{align}

The following definition will be used repeatedly in the proofs of the lemmas below. 

\begin{definition}\label{def:fgamma}
	\rm{
	Let $N\in \N$ and $\gamma:\{0,\ldots, N\}\to \Z^d$ be a finite path in $\Z^d$. We denote by $\cF_\gamma$ the range of a branching random walk indexed by a tree defined by taking a line of length $N+1$ and attaching trees to the left and the right of this line as follows: for $i\in \{0,\ldots, N\}$ the number of offspring of $i$ to the left (resp.\ to the right) of the line, denoted~$\ell_i$ (resp.\ $r_i$) have distribution
	\[
	\pr{\ell_i=\ell, r_i=r} = \mu(r+\ell+1)\quad  \text{ for } \ r,\ell\geq 0.
 	\]
 	The number of offspring of every other vertex which is not on the line is independently distributed according to $\mu$. 
	The root of the $i$-th tree is at location $\gamma(i)$ for $i=0,\ldots, N$ and we assign i.i.d.\ simple random walk increments to the edges of the attached trees. We write $\cF_{\gamma,-}$ for the range of the branching random walk indexed by the set of trees to the left of the line (including their roots) and excluding the first tree to the left. 
}
\end{definition}

\begin{remark}
	\rm{
	We note that in the definition above for every vertex $i$ on the line, the tree attached to $i$ to the left (or to the right) of the line has the distribution of an adjoint tree.
	}
\end{remark}

Let $\gamma:\{0,\ldots, N\}\to \Z^d$ be a finite path. We write $s(\gamma)$ for the probability that a simple random walk started from $\gamma(0)$ follows $\gamma$ for its first $N$ steps. 
	The following claim is a direct consequence of Lemma 6.1 in~\cite{Zhu}. We include a short proof here for the reader's convenience. 
	\begin{claim}\label{cl:fromzhu}
	Let $\epsilon<1/3$ and let $A\subseteq \Z^d$ be a finite subset. We then have  
	\begin{align*}
		\lim_{\|w\|\to\infty}\sup_{y\in\partial\ball{0}{\|w\|^{1-\epsilon}}}	\frac{1}{g(w)} \cdot \sum_{\gamma:y\to w}	s(\gamma) \cdot \pr{\cF_{\gamma}\cap A \neq \emptyset} = 0.
	\end{align*}
	\end{claim}

	\begin{proof}[\bf Proof]
		Writing $r_A(x)$ for the probability that an adjoint tree hits the set $A$ when it starts from $x$, we get 
		\begin{align*}
			\pr{\cF_{\gamma}\cap A\neq \emptyset} = 1- \prod_{i=1}^{|\gamma|} (1-r_{A}(\gamma(i))).
		\end{align*}
		By a union bound we get
		\begin{align}\label{eq:unionboudn}
			r_{A}(\gamma(i)) \lesssim |A| \cdot\max_{a\in A}g(\gamma(i)-a).
		\end{align}
Let $\delta>0$ be sufficiently small to be chosen later. Using the local CLT (see e.g.~\cite{LL10})), it is straightforward to check that as $\|x\|\to\infty$
		\[
		g(x) = (1+o(1)) \sum_{\gamma: |\gamma|\leq \|x\|^{2+\delta}} s(\gamma).
		\]
		Let $R=\|w\|^{1-\epsilon}$ and let $y\in \partial\ball{0}{R}$.  Then we have by the Markov property
		\begin{align*}
			\sum_{\substack{\gamma: y\to w\\ \gamma\cap \ball{0}{R^{1-\delta}}\neq \emptyset}}  s(\gamma) \lesssim \frac{1}{R^{\delta(d-2)}}\cdot  g(w). 
		\end{align*}
		For all $\gamma:y\to w$ such that $\gamma\cap \ball{0}{R^{1-\delta}}= \emptyset$, i.e.\ $\|\gamma(i)\|\geq R^{1-\delta}$ for all $i$, and $|\gamma|\leq \|w\|^{2\delta}$ we have using also~\eqref{eq:unionboudn}
			\begin{align*}
s(\gamma)	\cdot \pr{\cF_{\gamma}\cap A\neq \emptyset} &\leq s(\gamma) \cdot \left( 1-\left(1-\frac{c|A|}{R^{(1-\delta)(d-2)}}\right)^{\|w\|^{2+\delta}} \right) \\
			&\asymp s(\gamma) \cdot |A| \cdot \frac{\|w\|^{2+\delta}}{R^{(1-\delta)(d-2)}} = s(\gamma) \cdot |A|\cdot \frac{1}{\|w\|^{(1-\delta)(1-\epsilon)(d-2)-2-\delta}},
		\end{align*}
		where $c$ is a positive constant. 
		Taking the sum over all paths $\gamma$ such that $\gamma\cap \ball{0}{R^{1-\delta}}= \emptyset$ and  $|\gamma|\leq \|w\|^{2\delta}$ we get an upper bound of 
		\[
		g(w) \cdot |A| \cdot\frac{1}{\|w\|^{(1-\delta)(1-\epsilon)(d-2)-2-\delta}}.
		\]
		Therefore, overall we obtain for every $y\in \partial \ball{0}{R}$
		\begin{align*}
			\frac{1}{g(w)}\cdot \sum_{\gamma: y\to w}	s(\gamma)\cdot \pr{ \cF_{\gamma}\cap A\neq \emptyset}  \lesssim \frac{1}{R^{\delta(d-2)}}  +|A|\cdot  \frac{1}{\|w\|^{(1-\delta)(1-\epsilon)(d-2)-2-\delta}}.
		\end{align*}
		Taking $\delta$ sufficiently small so that the exponent above is strictly positive concludes the proof.
	\end{proof}

	\begin{lemma}\label{lem:hitfinitetoinfinite}
Let $A$ and $B$ be two finite sets in $\Z^d$. Then we have
		\[
		\lim_{\|w\|\to\infty} \frac{\pr{\cT_c^w\cap A\neq \emptyset, \cT_c^w\cap B\neq \emptyset}}{g(w)} = \sum_{x\in B} \pr{\cT^x\cap A\neq \emptyset, \cT_-^x\cap B=\emptyset}.
		\]
	\end{lemma}
\begin{remark}\rm{Note that specifying the result of the lemma to the case when $A= B$, we recover~\eqref{BCaphit}. } 
\end{remark}	
	\begin{proof}[\bf Proof of Lemma~\ref{lem:hitfinitetoinfinite}]
		
		We say that the walk indexed by $\mathcal T_c$ hits $B$ through a finite path $\gamma:\{0,\dots,N\} \to \mathbb Z^d$ and write $\Gamma=\gamma$, if the restriction 
		of the walk to the geodesic $\mathfrak g$ from the root to the first vertex (in the lexicographical order) at which the walk hits $B$ is $\gamma$ (here $N$ is the distance to the root from this vertex).	 Then by definition, the restriction of the walk to the set of trees to the left of this geodesic (excluding the last point) does not hit $B$. Let us denote by $\cL_{\gamma,-}$ the range of the walk indexed by this set of trees, including the geodesic $\mathfrak g$, but excluding its final point. We let $\cR_{\gamma}$ be the range of the walk indexed by all the trees attached to the vertices on the geodesic. We then have 
		\begin{align*}
			\pr{\cT_c^w\cap A\neq \emptyset, \cT_c^w\cap B\neq \emptyset}& = \sum_{x\in B} \sum_{\gamma: w\to x} \pr{\Gamma=\gamma, \cL_{\gamma,-}\cap B=\emptyset, \cR_\gamma\cap A\neq \emptyset}\\&=  \sum_{x\in z+B} \sum_{\gamma: w\to x} s(\gamma)\cdot \pr{ \cL_{\gamma,-}\cap B=\emptyset, \cR_\gamma\cap A\neq \emptyset}
		\end{align*}
		We recall that if $v$ is a vertex on the geodesic $\mathfrak g$ (excluding its last vertex), then its number of offspring is distributed according to $\mu(\cdot +1)$. Moreover, if $\ell_v$ (resp.\ $r_v$) is its number of offspring to the left (resp.\ right) of $\mathfrak g$, then for all $\ell, r\geq 0$		\begin{align}\label{eq:distrofoffs}
			\pr{\ell_v=\ell, r_v=r} = \mu(\ell+r+1).
		\end{align}
		Writing $\stackrel{\leftarrow}{\gamma}$ for the reversal of the path $\gamma$, i.e.\ $\stackrel{\leftarrow}{\gamma}=(\gamma(N-i))_{0\leq i \leq N}$, and recalling Definition~\ref{def:fgamma} we see that 
				\[
		\pr{ \cL_{\gamma,-}\cap B=\emptyset, \cR_\gamma\cap A\neq \emptyset} = \pr{ \cF_{\stackrel{\leftarrow}{\gamma},-}\cap B=\emptyset, \cF_{\stackrel{\leftarrow}{\gamma}}\cap A\neq \emptyset}.	
		\]
Using this and the reversibility of simple random walk on $\Z^d$ we obtain 
 	\begin{align*}
 		\sum_{\gamma:w\to x} s(\gamma) \cdot \pr{ \cL_{\gamma,-}\cap B=\emptyset, \cR_\gamma\cap A\neq \emptyset} &= \sum_{\gamma:w\to x} s(\stackrel{\leftarrow}{\gamma}) \cdot  \pr{ \cF_{\stackrel{\leftarrow}{\gamma},-}\cap B=\emptyset, \cF_{\stackrel{\leftarrow}{\gamma}}\cap A\neq \emptyset} \\ 
 		&= \sum_{\gamma:x\to w} s(\gamma) \cdot \pr{ \cF_{{\gamma},-}\cap B=\emptyset, \cF_{{\gamma}}\cap A\neq \emptyset}. 
  	\end{align*}
		Let $R=\|w\|^{1-\epsilon}$ with $\epsilon<1/3$ as in Claim~\ref{cl:fromzhu}. Take $\|w\|$ sufficiently large so that the ball $\ball{0}{R}$ contains both $A$ and $B$. By considering the first time the path $\gamma$ hits $\ball{0}{R}$ we get for $x\in B$
		\begin{align*}
			&\sum_{\gamma:x\to w} s(\gamma) \cdot \pr{ \cF_{{\gamma},-}\cap B=\emptyset, \cF_{{\gamma}}\cap A\neq \emptyset}\\
			&= \sum_{y\in \partial\ball{0}{R}} 
			\sum_{\substack{\gamma_1:x\to y\\ 				\gamma_1\setminus \{y\}\cap \partial\ball{0}{R}=\emptyset}} \sum_{\gamma_2:y\to w} s(\gamma_1)\cdot s(\gamma_2)\cdot \pr{( \cF_{\gamma_1,-}\cup \cF_{\gamma_2,-})\cap B=\emptyset,  (\cF_{\gamma_1}
				\cup \cF_{\gamma_2}) \cap A\neq \emptyset}.
		\end{align*}
	Using Claim~\ref{cl:fromzhu} and the fact that $\sum_{\gamma:y\to w} s(\gamma)=g(y-w)=g(w)(1+o(1))$ for $y\in \partial\ball{0}{R}$ we obtain from the above
	\begin{align*}
&\lim_{\|w\|\to\infty} 	\frac{1}{g(w)}\sum_{\gamma:x\to w} s(\gamma)  \cdot \pr{ \cF_{{\gamma},-}\cap B=\emptyset, \cF_{{\gamma}}\cap A\neq \emptyset} \\
&=\lim_{\|w\|\to\infty} \sum_{y\in \partial\ball{0}{R}} \sum_{\substack{\gamma_1:x\to y\\ 
		\gamma_1\setminus \{y\}\cap \partial\ball{0}{R}=\emptyset}} s(\gamma_1)\cdot \pr{ \cF_{\gamma_1,-}
	\cap B=\emptyset,  \cF_{\gamma_1}	\cap A\neq \emptyset}. 
	\end{align*}
	Therefore, so far we have established that 
		\begin{align}\label{eq:equalityforcritical}
			\nonumber&\lim_{\|w\|\to\infty} \frac{\pr{\cT_c^w\cap A\neq \emptyset, \cT_c^w\cap B\neq \emptyset}}{g(w)} \\ &= \lim_{\|w\|\to\infty}\sum_{\substack{x\in B\\y\in \partial\ball{0}{R}}}\sum_{\substack{\gamma:x\to y\\ 
					\gamma\setminus \{y\}\cap \partial\ball{0}{R}=\emptyset}} s(\gamma)\cdot \pr{ \cF_{\gamma}
					\cap A\neq \emptyset, \cF_{\gamma,-}
				\cap B=\emptyset}. 
		\end{align}
		To conclude the proof it suffices to show that 
		\begin{align}\label{eq:finalgoal}\begin{split}
			\lim_{\|w\|\to\infty}\sum_{\substack{x\in B\\y\in \partial\ball{0}{R}}}\sum_{\substack{\gamma:x\to y\\ 
					\gamma\setminus \{y\}\cap \partial\ball{0}{R}=\emptyset}} s(\gamma)\cdot \pr{ \cF_{\gamma}
					\cap A\neq \emptyset, \cF_{\gamma,-}
				\cap B=\emptyset}\\
				 = \sum_{x\in B} \pr{\cT^x\cap A\neq \emptyset, \cT^x_-\cap B=\emptyset}.
\end{split}
			\end{align}
		Let $X$ be the simple random walk performed by the spine of $\cT$. 
		By considering the first visit to~$\partial \ball{0}{R}$ by $X$ denoted by $\tau_{\partial\ball{0}{R}}$  and recalling Definition~\ref{def:fgamma} we get for $x\in B$
		\begin{align*}
			&\pr{\cT^x\cap A\neq \emptyset, \cT^x_-\cap B=\emptyset} \\
			&= \sum_{y\in \partial \ball{0}{R}}\sum_{\gamma:x\to y} \prstart{X[0,\tau_{\partial \ball{0}{R}}]=\gamma}{x} \pr{(\cF_{\gamma,-}\cup\cT_-^y)\cap B=\emptyset, (\cF_{\gamma}\cup \cT^y)\cap A\neq \emptyset},
				\end{align*}
			where $\cT^y$ (resp.\ $\cT^y_-$) denotes the range of the branching random walk (resp.\ the past) after the spine reaches $y$ for the first time. 
			
		Now it is also known from~\cite{Zhu} that for some constant $C>0$, for all $R$ large enough, one has 
		$$\sup_{y\in \partial B(0,R)} \mathbb P(\mathcal T^y \cap (A \cup B)\neq \emptyset) \le C \cdot \frac{\textrm{BCap}(A \cup B)}{R^{d-4}},$$
		which implies that given $A,B\subset \mathbb Z^d$ (recall that $d\ge 5$), 
		\begin{equation}\label{eq:BCap4}
			\lim_{\|w\|\to \infty} \sup_{y\in \partial B(0,R)} \mathbb P(\mathcal T^y \cap (A \cup B)\neq \emptyset) = 0. 
		\end{equation}
		Therefore, we conclude 
		\begin{align*}
			&\pr{\cT^x\cap A\neq \emptyset, \cT^x_-\cap B=\emptyset} \\
			&=\lim_{\|w\|\to\infty}\sum_{y\in \partial \ball{0}{R}}\sum_{\gamma:x\to y} \prstart{X[0,\tau_{\partial \ball{0}{R}}]=\gamma}{x} \pr{\cF_{\gamma,-} \cap B=\emptyset, \cF_{\gamma}\cap A\neq \emptyset}.	
		\end{align*}
		Finally, taking the sum over all $x\in z+B$ proves~\eqref{eq:finalgoal}. 
	\end{proof}

	\begin{lemma}\label{lem:inftobcap}
		We have that 
		\[
		\lim_{\|z\|\to\infty} \frac{1}{G(z)}\cdot \sum_{x\in z+B} \pr{\cT^x\cap A\neq \emptyset, \cT_-^x\cap (z+B)= \emptyset} = 2\cdot \bcap{A}\cdot \bcap{B}.
		\]
	\end{lemma}
	
	Before proceeding to the proof, we state two results that we prove afterwards.

		\begin{lemma}\label{lem:BCap2}
		Let $A$ and $B$ be finite subsets of $\mathbb Z^d$. Then 
		$$\lim_{r\to \infty}  \limsup_{\|z\| \to \infty}  \sup_{y\in \partial B(z,r)} \frac{\mathbb P(\mathcal T^y \cap A \neq \emptyset, \mathcal T^y_-\cap (z+B)  \neq \emptyset)}{G(z)} = 0 . $$ 
	\end{lemma}

		\begin{lemma}\label{claim.pastfuturehit} Let $A$ be a finite subset of $\mathbb Z^d$. Then 
		$$\lim_{\|z\|\to \infty} \frac{\mathbb P(\mathcal T^z \cap A \neq \emptyset) }{G(z)}  =2\cdot \bcap{A}. $$
	\end{lemma}

	\begin{proof}[\bf Proof of Lemma~\ref{lem:inftobcap}]
		Let $r$ be sufficiently large so that $A,B\subseteq \ball{0}{r}$. Take $\|z\|>2r$. 
		With the same notation as before, i.e.\ writing $X$ for the random walk that the spine performs in its natural parametrisation and denoting by $\tau_{z,r}$ the first time that $X$ hits $\partial \ball{z}{r}$, we get using Lemma~\ref{lem:BCap2} that
		\begin{align*}
			&\frac{\pr{\cT^x\cap A\neq \emptyset, \cT_-^x\cap (z+B)= \emptyset}}{G(z)} \\
			&= (1+o(1))\cdot\frac{1}{G(z)}\cdot \sum_{y\in \partial \ball{z}{r}} \sum_{\gamma: x\to y} \pr{X[0,\tau_{{z},{r}}]=\gamma, \cF_\gamma\cap A\neq \emptyset, \cF_{\gamma,-}\cap (z+B)=\emptyset}\\
			&=(1+o(1))\cdot \sum_{y\in \partial \ball{z}{r}} \sum_{\gamma: x\to y} \pr{X[0,\tau_{{z},{r}}]=\gamma} \pr{\cF_{\gamma,-}\cap(z+B)=\emptyset}\cdot  \left( \frac{\pr{\cT_-^y\cap A\neq \emptyset}}{G(z)} + o(1)\right).
		\end{align*}
		Using again~\eqref{eq:BCap4} we see that 
			$$ \lim_{r\to \infty} \sum_{\gamma : x \to \partial B(z,r)}\pr{X[0,\tau_{{z},{r}}]=\gamma}  \cdot \mathbb P(\mathcal F_{\gamma,-} \cap (z+B)= \emptyset) = \mathbb P(\mathcal T^x_-\cap (z+B)=\emptyset).$$
		Therefore, from this and Lemma~\ref{claim.pastfuturehit} we deduce
		\begin{align*}
			\lim_{\|z\|\to\infty} \frac{\pr{\cT^x\cap A\neq \emptyset, \cT_-^x\cap (z+B)= \emptyset}}{G(z)} &= 2\cdot \bcap{A} \cdot \lim_{\|z\|\to\infty}\sum_{x\in z+B} \pr{\cT_-^x\cap (z+B)=\emptyset} \\
			&= 2\cdot \bcap{A} \cdot \bcap{B}
		\end{align*}
		and this concludes the proof.
	\end{proof}

	\begin{proof}[\bf Proof of~\eqref{limit.BCap}]
		
		The proof follows by combining~\eqref{eq:incexcl} with Lemmas~\ref{lem:hitfinitetoinfinite} and~\ref{lem:inftobcap}.
	\end{proof}

		\begin{proof}[\bf Proof of Lemma~\ref{lem:BCap2}]
		Let $r>0$, $z\in \mathbb Z^d$, and $y\in \partial B(0,r)$ be given. We first label the elements of the spine of $\mathcal T$ together with the root by integers, according to their distance to the root, with the root having label $0$. Call $I$ the smallest label such that the walk hits $A$ on one of the two trees attached to the vertex with label $I$, and $J$ the smallest label such that the walk hits $z+B$ on the tree attached to the vertex with label $J$ on the left of the spine. We distinguish three cases. Either $I<J$, $I>J$, or $I=J$, resulting in the following bound, with $x\in \mathbb Z^d$ standing for the position of the walk on the spine at the vertex with label $\min(I,J)$,  
		\begin{align}\label{lem:hitAB1}
			\nonumber & \mathbb P(\mathcal T^y \cap A \neq \emptyset, \mathcal T^y_-\cap (z+B)\neq \emptyset) 
			\le \sum_{x\in \mathbb Z^d} g(y-x)\cdot  \Big( \mathbb P(\widehat {\mathcal T}_c^x\cap A \neq \emptyset) \cdot \mathbb P(\mathcal T^x_-\cap (z+B)\neq \emptyset) \\
			& \qquad + \mathbb P(\widetilde {\mathcal T}_c^x\cap (z+B) \neq \emptyset) \cdot \mathbb P(\mathcal T^x\cap A\neq \emptyset) 
			+ \mathbb P(\widehat {\mathcal T}_c^x\cap A \neq \emptyset, \widetilde{ \mathcal T}_c^x \cap (z+B)\neq \emptyset)\Big), 
		\end{align}
		where in the last probability, the two underlying trees $\widehat {\mathcal T}_c$ and $\widetilde{ \mathcal T}_c$ are not independent: the former is the union of the latter, together with another copy, sharing the same root, which is correlated to the first one only through the number of offspring of the root in these two trees. The two first terms on the right hand side of~\eqref{lem:hitAB1} are handled using that by~\cite{Zhu}, for some constant $C>0$ (depending on $A$ and $B$), uniformly in $x\in \mathbb Z^d$, 
		$$\mathbb P(\widehat {\mathcal T}_c^x\cap A \neq \emptyset)\le C\cdot g(x),\quad \textrm{and}\quad \mathbb P(\widetilde {\mathcal T}_c^x\cap (z+B) \neq \emptyset) \le C \cdot g(z-x), $$
		$$\mathbb P(\mathcal T^x\cap A \neq \emptyset)\le C\cdot G(x),\quad \textrm{and}\quad \mathbb P(\mathcal T_-^x\cap (z+B) \neq \emptyset) \le C \cdot G(z-x). $$ 
		Moreover, a direct computation (see also Lemma 2.3 in~\cite{AS24}) shows that for some constant $C>0$, for any $r>0$, any $z$ with $\|z\|>2r$, and any $y\in \partial B(z,r)$, 
		$$\sum_{x\in \mathbb Z^d} g(x-y) \big(g(x) G(z-x ) + G(x) g(z-x)\big) \le C \cdot \frac{G(z)}{r^{d-4}}. $$ 
		Hence, we only need to bound the last probability term in~\eqref{lem:hitAB1}. As was recalled above $\widehat{\mathcal T}_c$ there is the union of two trees, say $\widetilde{\mathcal T}_{c,1}$ and $\widetilde{\mathcal T}_{c,2}$, which are copies of $\widetilde{\mathcal T}_c$, sharing the same root, and 
		correlated only through the  number of offspring of the root in the two trees. More precisely, the probability that the root has $i$ children in the first tree and $j$ in the second one is equal to $\mu(i+j+1)$. One then has 
		\begin{align}\label{lem:hitAB2}
			\nonumber & \mathbb P(\widehat {\mathcal T}_c^x\cap A \neq \emptyset, \widetilde{ \mathcal T}_c^x \cap (z+B)\neq \emptyset)\\
			&\quad \le \mathbb P(\widetilde{\mathcal T}_{c,1}^x \cap A \neq \emptyset, \widetilde{\mathcal T}_{c,2}^x \cap (z+B)\neq \emptyset)
			+ \mathbb P(\widetilde{\mathcal T}_c^x \cap A \neq \emptyset, \widetilde{\mathcal T}_c^x \cap (z+B)\neq \emptyset). 
		\end{align} 
		Concerning the first term on the right hand side above, one can condition first on the number of offspring of the root on both trees $\widetilde{\mathcal T}_{c,1}$ and $\widetilde{\mathcal T}_{c,2}$, and use a union bound on all trees emanating from the children of the root. Recalling that $\mu$ is assumed to have a finite third moment, we obtain that this first term is bounded by a constant (depending on $A$ and $B$) times $g(x) g(z-x)$, and hence can be handled as above. As for the second term on the right-hand side of~\eqref{lem:hitAB2}, we use a standard second moment bound. 
		Call $u_1$ the first vertex (for the lexicographical order) of $\widetilde{\mathcal T}_c$ at which the walk hits $A$ and $u_2$ the first one at which the walk hits $z+B$. Summing over all possible locations~$x'$ of the walk at the most recent common ancestor of $u_1$ and $u_2$, we get that 
		for some constant $C>0$, 
		\begin{align*}
			\mathbb P(\widetilde{\mathcal T}_c^x \cap A \neq \emptyset, \widetilde{\mathcal T}_c^x \cap (z+B)\neq \emptyset)
			& \le C \sum_{x'\in \mathbb Z^d} g(x'-x) g(x') g(x'-z)\\
			& \le C\big(g(x)G(z-x) + G(x)g(z-x)\big),
		\end{align*} 
		and we conclude the proof using the same argument as above. 
	\end{proof}

	\begin{proof}[\bf Proof of Lemma~\ref{claim.pastfuturehit}]
		Denote by $\mathcal T_+$ the set of vertices of $\mathcal T$ which are not in $\mathcal T_-$. Since the walk indexed by the spine of $\mathcal T$ or by the critical tree attached to the right of the root, have much smaller chance to hit a given set from far away than the whole walk indexed by $\mathcal T_-$, we can deduce from~\eqref{BCaphit} that one also has for any finite set $A$, 
		$$\textrm{BCap}(A) = \lim_{\|z\|\to \infty} \frac{\mathbb P(\mathcal T_+^z\cap A\neq \emptyset)}{G(z)}. $$ 
		Hence it just amounts to showing that 
		$$\lim_{\|z\|\to \infty} \frac{\mathbb P(\mathcal T_+^z\cap A\neq \emptyset,\mathcal T_-^z\cap A\neq \emptyset)}{G(z)} = 0. $$ 
		Using a similar argument and the same notation as in the proof of Lemma~\ref{lem:BCap2}, we get 
		\begin{align*}
			& \mathbb P(\mathcal T_+^z\cap A\neq \emptyset,\mathcal T_-^z\cap A\neq \emptyset) \\
			&\le \sum_{x\in \mathbb Z^d} 
			g(x-z) \cdot \Big( \mathbb P(\mathcal T^x_-\cap A \neq \emptyset)\cdot \mathbb P(\widetilde{\mathcal T}^x_c\cap A \neq \emptyset) 
			+ \mathbb P(\mathcal T^x_+\cap A \neq \emptyset)\cdot \mathbb P(\widetilde{\mathcal T}^x_c\cap A \neq \emptyset) \\
			& \quad + \mathbb P(\widetilde{\mathcal T}_{c,1}^x \cap A \neq \emptyset, \widetilde{\mathcal T}_{c,2}^x \cap A\neq \emptyset)\Big) \\
			& \le C\sum_{x\in \mathbb Z^d} g(x-z) \big(G(x) g(x) + g(x)^2\big) \le C \cdot \|z\|^{3-d}, 
		\end{align*}
		which concludes the proof, thanks to~\eqref{asympG}. 
	\end{proof}

\end{document}